\documentclass[11pt]{amsart}
\usepackage{amssymb}

\newtheorem{theorem}{Theorem}[section]

\newtheorem{lemma}[theorem]{Lemma}
\newtheorem{proposition}[theorem]{Proposition}
\theoremstyle{definition}

\theoremstyle{remark}
\newtheorem{remark}[theorem]{Remark}

\numberwithin{equation}{section}

\begin{document}

\title{Borel equivalence relations between $\ell_1$ and $\ell_p$}
\author{Longyun Ding}
\address{School of Mathematical Sciences and LPMC, Nankai University, Tianjin, 300071, P.R.China}
\email{dinglongyun@gmail.com}
\thanks{Research partially supported by the National Natural Science Foundation of China (Grant No. 10701044) and Program for New Century Excellent Talents in University.}
\author{Zhi Yin}
\email{will.yin@hotmail.com}


\date{\today}

\begin{abstract}
In this paper, we show that, for each $p>1$, there
are continuum many Borel equivalence relations between $\Bbb R^\omega/\ell_1$ and $\Bbb R^\omega/\ell_p$ ordered by $\le_B$ which are pairwise Borel incomparable.
\end{abstract}
\maketitle

\section{Introduction}

A {\it Polish space} is a topological space that admits a compatible complete separable metric. For more details in descriptive set theory, one can see \cite{kechris}.
Let $X,Y$ be Polish spaces, $E,F$ equivalence relations on $X,Y$ respectively, we say $E$ is {\it Borel reducible} to $F$, denoted by $E\le_B F$, if there exists a
Borel function $\theta:X\to Y$ satisfying
$$xE\hat x\iff\theta(x)F\theta(\hat x).$$
We say $E$ is {\it strictly Borel reducible} to $F$, $E<_B F$ in notation, if $E\le_B F$ but $F\not\le_B E$. We refer to \cite{gaobook} for background on Borel reducibilty.

R. Dougherty and G. Hjorth \cite{DH} proved that, for $p,q\ge 1$,
$$\Bbb R^\omega/\ell_p\le_B\Bbb R^\omega/\ell_q\iff p\le q.$$
A question of S. Gao in \cite{gao} asking whether $\Bbb R^\omega/\ell_p$ is the greatest lowest bound of $\{\Bbb R^\omega/\ell_q:p<q\}$. T. M\' atrai answer this question
in the negative by showing, for $1\le p<q$, every linear order which embeds into $(P(\omega)/{\rm fin},\subset)$ also embeds into the set of equivalence relations
between $\Bbb R^\omega/\ell_p$ and $\Bbb R^\omega/\ell_q$ ordered by $<_B$ (see \cite{matrai}, Corollary 31).

We can see that all equivalence relations considered in M\' arai's paper \cite{matrai}
are pairwise Borel comparable. A question arises naturally that, for $1\le p<q$, whether there are equivalence relations $E,F$ such that
$\Bbb R^\omega/\ell_p\le_B E,F\le_B\Bbb R/\ell_q$ but $E,F$ are incomparable. Both Gao and M\' atrai asked this question in the special case $p=1,q=2$.
In this paper, we show that, for each $p>1$, there are continuum many pairwise Borel incomparable equivalence relations between $\Bbb R^\omega/\ell_1$
and $\Bbb R^\omega/\ell_p$.

\section{Some notes on {\bf E}$_f$}

We denote by $\Bbb R^+$ the set of nonnegative real numbers. Let $f:[0,1]\to\Bbb R^+$. M\' atrai \cite{matrai} defined the relation {\bf E}$_f$ on $[0,1]^\omega$ by setting,
for every $(x_n)_{n<\omega},(y_n)_{n<\omega}\in[0,1]^\omega$,
$$(x_n){\textbf E}_f(y_n)\iff\sum_{n<\omega}f(|y_n-x_n|)<\infty.$$
It is straightforward that {\bf E}$_f$ is a Borel relation whenever $f$ is Borel.

The following proposition answers when {\bf E}$_f$ is an equivalence relation.

\begin{proposition}[M\' atrai \cite{matrai}, Proposition 2] \label{equivalence}
Let $f:[0,1]\to\Bbb R^+$ be a bounded function. Then {\bf E}$_f$ is an equivalence relation iff the following conditions hold:
\begin{enumerate}
\item[(R$_1$)] $f(0)=0;$
\item[(R$_2$)] there is a $C\ge 1$ such that for every $x,y\in[0,1]$ with $x+y\in[0,1]$,
$$f(x+y)\le C(f(x)+f(y)),$$
$$f(x)\le C(f(x+y)+f(y)).$$
\end{enumerate}
\end{proposition}

A nonreducibility result was obtained in \cite{matrai} for a class of {\bf E}$_f$'s as follows.

\begin{theorem}[M\' atrai \cite{matrai}, Theorem 18] \label{inreduce}
Let $1\le\alpha<\infty$ and let $\varphi,\psi:[0,1]\to[0,+\infty)$ be continuous. Set $f(x)=x^\alpha\varphi(x),g(x)=x^\alpha\psi(x)$ for $x\in[0,1]$
and suppose that $f,g$ are bounded and {\bf E}$_f$ and {\bf E}$_g$ are equivalence relations. Suppose $\psi(x)>0\,(x>0)$, and
\begin{enumerate}
\item[(A$_1$)] there exist $\varepsilon>0,M<\omega$ such that for every $n>M$ and $x,y\in[0,1]$,
$$\varphi(x)\le\varepsilon\varphi(y)\varphi(1/2^n)\Rightarrow x\le\frac{y}{2^{n+1}};$$
\item[(A$_2$)] $\lim_{n\to\infty}\psi(1/2^n)/\varphi(1/2^n)=0$.
\end{enumerate}
Then {\bf E}$_g\not\le_B${\bf E}$_f$.
\end{theorem}

\begin{remark} \label{incomparable}
We may replace condition (A$_2$) in the theorem by
\begin{enumerate}
\item[(A$_2$)'] $\liminf_{n\to\infty}\psi(1/2^n)/\varphi(1/2^n)=0$.
\end{enumerate}
In fact, we can check that the proof for Theorem 18 of \cite{matrai} is still valid under condition (A$_2$)'.
In this paper, condition (A$_2$)' is the key to prove incomparability between equivalence relations.
\end{remark}

Mostly, we focus on equivalence relations {\bf E}$_f$ in which $f(x)=x^\alpha\varphi(x)$ for $x\in[0,1]$ with $\varphi$ continuous increasing.

\begin{lemma} \label{varphi}
Let $\alpha\ge 1$ and $\varphi:[0,1]\to[0,\infty)$ be an increasing function with $\varphi(1/2)>0$. Set $f(x)=x^\alpha\varphi(x)$ for $x\in[0,1]$.
If there exists $\delta>0$ such that, for each $n>1$,
$$\varphi(1/2^n)\ge\max_{1\le i\le n-1}\{\delta\varphi(1/2^i)\varphi(1/2^{n-i})\},$$
then {\bf E}$_f$ is an equivalence relation and condition (A$_1$) in Theorem \ref{inreduce} holds.
\end{lemma}

\begin{proof}
Note that for $n>1$ we have $\varphi(1/2^n)\ge\delta\varphi(1/2)\varphi(1/2^{n-1})$. Since $\varphi(1/2)>0$ and $\varphi$ is increasing, we have $\varphi(x)>0$ for $x>0$.
By Proposition \ref{equivalence} and Theorem \ref{inreduce}, we need only to check (R$_1$), (R$_2$), and (A$_1$).

For (R$_1$), $f(0)=0$ is trivial.

For (R$_2$), let $x,y\in[0,1]$ with $x+y\in[0,1]$. Without loss of generality, we can assume that $x\ge y>0$. Since $f(x)=x^\alpha\varphi(x)$ is increasing, we have
$$f(x)\le f(x+y)\le(f(x+y)+f(y)).$$

If $x>1/4$, then
$$f(x+y)\le f(1)=\frac{4^\alpha\varphi(1)}{\varphi(1/4)}f(1/4)\le\frac{4^\alpha\varphi(1)}{\varphi(1/4)}(f(x)+f(y)).$$

If $x\le 1/4$, let $x\in(1/2^{n+1},1/2^n]$ with $n>1$. Then
$$f(x+y)\le f(2x)\le f(1/2^{n-1})=\frac{1}{2^{(n-1)\alpha}}\varphi(1/2^{n-1}),$$
$$f(x)\ge f(1/2^{n+1})=\frac{1}{2^{(n+1)\alpha}}\varphi(1/2^{n+1})\ge\frac{\delta}{2^{(n+1)\alpha}}\varphi(1/4)\varphi(1/2^{n-1}).$$
Thus we have
$$f(x+y)\le\frac{4^\alpha}{\delta\varphi(1/4)}(f(x)+f(y)).$$
Therefore, $C=\max\left\{1,\frac{4^\alpha\varphi(1)}{\varphi(1/4)},\frac{4^\alpha}{\delta\varphi(1/4)}\right\}$ witnesses that (R$_2$) holds.

For (A$_1$), fix a $0<\varepsilon<\min\{1/\varphi(1),\delta\varphi(1/4)/\varphi(1),\delta^2\varphi(1/4)\}$.
For $x,y\in[0,1]$ and $n>0$, assume for contradiction that
$$\varphi(x)\le\varepsilon\varphi(y)\varphi(1/2^n),\mbox{ but }x>\frac{y}{2^{n+1}}.$$

If $y=0$, since $\varphi(x)\le\varepsilon\varphi(0)\varphi(1/2^n)\le\varphi(0)$, we have $x=0$. It contradict to $x>\frac{y}{2^{n+1}}$.

If $y>0$, let $y\in(1/2^{m+1},1/2^m]$ for some $m\in\omega$, then $x>1/2^{m+n+2}$. If $m=0$, we have
$$\varepsilon\varphi(1)\varphi(1/2^n)\ge\varphi(x)\ge\varphi(1/2^{n+2})\ge\delta\varphi(1/4)\varphi(1/2^n),$$
contradicting $\varepsilon<\delta\varphi(1/4)/\varphi(1)$. If $m\ge 1$, we have
$$\begin{array}{ll}\varepsilon\varphi(1/2^m)\varphi(1/2^n)&\ge\varphi(x)\ge\varphi(1/2^{m+n+2})\cr
&\ge\delta\varphi(1/2^{m+2})\varphi(1/2^n)\cr
&\ge\delta^2\varphi(1/4)\varphi(1/2^m)\varphi(1/2^n),\end{array}$$
contradicting $\varepsilon<\delta^2\varphi(1/4)$.
\end{proof}

\section{pairwise incomparable equivalence relations}

From Lemma \ref{varphi}, we can define $\varphi$ from a decreasing sequence $(u_n)_{n<\omega}$ by setting $\varphi(1/2^n)=u_n$ and then extend $\varphi$ to $[0,1]$ to be
a continuous increasing function which is affine on each $[1/2^{n+1},1/2^n]$.

\begin{lemma} \label{un}
Let $0<\delta,\lambda<1$ and $u_0=u_1=1$. For $n>1$, suppose that $u_n=u_{n-1}$ or $u_n=\lambda u_{n-1}+(1-\lambda)\max_{1\le i\le n-1}\{\delta u_iu_{n-i}\}$.
Then we have, for each $n>1$,
$$u_{n-1}\ge u_n\ge\max_{1\le i\le n-1}\{\delta u_iu_{n-i}\}.$$
\end{lemma}

\begin{proof}
We argue by induction on $n$. If $n=2$, then $u_2=u_1$ or $u_2=\lambda u_1+(1-\lambda)\delta u_1^2$. So $u_1\ge u_2\ge\delta u_1^2$.

For $n>2$, by induction hypothesis, $u_{k-1}\ge u_k\ge\max_{1\le i\le k-1}\{\delta u_iu_{k-i}\}$ for each $2\le k<n$. Thus
$$u_{n-1}\ge\max_{1\le i\le n-2}\{\delta u_iu_{n-i-1}\}\ge\max_{1\le i\le n-2}\{\delta u_iu_{n-i}\}.$$
Note that $u_{n-1}\ge\delta u_{n-1}u_1$, we have $u_{n-1}\ge\max_{1\le i\le n-1}\{\delta u_iu_{n-i}\}$. Then by the definition of $u_n$,
$$u_{n-1}\ge u_n\ge\max_{1\le i\le n-1}\{\delta u_iu_{n-i}\}.$$
\end{proof}

\begin{lemma} \label{reduce}
Let $\beta>\alpha\ge 1,0<\delta<1$ and $\lambda=2^{\alpha-\beta}$. Suppose that $(u_n)_{n<\omega}$ is a sequence as in Lemma \ref{un} and
$\varphi:[0,1]\to[0,\infty)$ is a continuous increasing function with $\varphi(1/2^n)=u_n$ for each $n<\omega$. Set $f(x)=x^\alpha\varphi(x)$ for $x\in[0,1]$.
Then {\bf E}$_f$ is an equivalence relation and
$$\Bbb R^\omega/\ell_1\le_B{\bf E}_f\le_B\Bbb R^\omega/\ell_\beta.$$
\end{lemma}

\begin{proof}
(1) From Lemma \ref{un}, we have
$$\varphi(1/2^n)=u_n\ge\max_{1\le i\le n-1}\{\delta u_iu_{n-i}\}=\max_{1\le i\le n-1}\{\delta\varphi(1/2^i)\varphi(1/2^{n-i})\}.$$
Thus by Lemma \ref{varphi}, {\bf E}$_f$ is an equivalence realtion.

(2) Fix a bijection $\langle\cdot,\cdot,\cdot\rangle:\{0,1\}\times\omega\times\omega\to\omega$.
For each $n\in\omega$, find a $c_n\in[0,1]$ such that $0<f(c_n)<1/2^n$.
We define $\theta_1:\Bbb R^\omega\to[0,1]^\omega$ by, for $(x_n)_{n<\omega}\in\Bbb R^\omega$, setting $\theta_1((x_n)_{n<\omega})=(y_m)_{m<\omega}$ with
$$\begin{array}{ll}y_m=c_n\iff &m=\langle 0,n,k\rangle,x_n\ge 0,k<[x_n/f(c_n)],\cr
&\mbox{ or }m=\langle 1,n,k\rangle,x_n<0,k<[-x_n/f(c_n)],\end{array}$$
and $y_m=0$ otherwise. It is easy to check that $\theta_1$ is Borel. For $(x_n)_{n<\omega},(\hat x_n)_{n<\omega}\in\Bbb R^\omega$, if
$\theta_1((x_n)_{n<\omega})=(y_m)_{m<\omega},\theta_1((\hat x_n)_{n<\omega})=(\hat y_m)_{m<\omega}$, we have
$$|x_n-\hat x_n|-1/2^{n-1}<\sum f(|y_m-\hat y_m|)<|x_n-\hat x_n|+1/2^{n-1},$$
where $\sum$ ranges over $\{m=\langle i,k,n\rangle:y_m\ne\hat y_m,k<\omega,i=0,1\}$. Thus
$$\sum_{n<\omega}|x_n-\hat x_n|<\infty\iff\sum_{m<\omega}f(|y_m-\hat y_m|)<\infty.$$
Therefore, $\theta_1$ witnesses that $\Bbb R^\omega/\ell_1\le_B{\bf E}_f$.

(3) For proving ${\bf E}_f\le_B\Bbb R^\omega/\ell_\beta$, by Theorem 16 of \cite{matrai}, we need only to find a function $\kappa:\{1/2^i:i<\omega\}\to[0,1]$ and
$L\ge 1$ satisfying that, for each $n<\omega$,
\begin{enumerate}
\item[(i)] $f(1/2^n)=\sum_{i=0}^n(\kappa(1/2^i)/2^{n-i})^\beta$;
\item[(ii)] $\sum_{i=n}^\infty\kappa(1/2^i)^\beta\le L\sum_{i=0}^n(\kappa(1/2^i)/2^{n-i})^\beta$;
\item[(iii)] $\kappa(1/2^n)\le L\cdot\max\{\kappa(1/2^i)/2^{n-i}:i<n\}$.
\end{enumerate}

To satisfy (i), we shall let $\kappa(1)=f(1)=u_0=1$ and, for $n>0$,
$$\kappa(1/2^n)^\beta=f(1/2^n)-f(1/2^{n-1})/2^\beta=(u_n-\lambda u_{n-1})/2^{n\alpha}.$$
Note that $u_1-\lambda u_0=1-\lambda\in[0,1]$ and, for $n>1$,
$$(1-\lambda)u_{n-1}\ge u_n-\lambda u_{n-1}\ge(1-\lambda)\max_{1\le i\le n-1}\{\delta u_iu_{n-i}\},$$
so $u_n-\lambda u_{n-1}\in[0,1]$. We see that $\kappa(1/2^n)$ is well defined.

Let $L=\max\{\sum_{k=0}^\infty 2^{-k\alpha},2,(\delta 2^\alpha)^{-1/\beta}\}$.

By the definition of $\kappa$, we have $\kappa(1/2^n)^\beta\le f(1/2^n)=\varphi(1/2^n)/2^{n\alpha}$. Hence
$$\sum_{i=n}^\infty\kappa(1/2^i)^\beta\le\sum_{i=n}^\infty\varphi(1/2^i)/2^{i\alpha}\le\sum_{i=n}^\infty\varphi(1/2^n)/2^{i\alpha}
=f(1/2^n)\sum_{i=n}^\infty\frac{1}{2^{(i-n)\alpha}}.$$
From (i), we know (ii) is satisfied.

For (iii), if $n=1$, then $\kappa(1/2)\le 1\le L\kappa(1)/2$.

Note that for each $n>1$, we have
$$\begin{array}{ll}\kappa(1/2^n)^\beta&=(u_n-\lambda u_{n-1})/2^{n\alpha}\le(1-\lambda)u_{n-1}/2^{n\alpha}\cr
&\le(1-\lambda)u_{n-2}/2^{n\alpha}\le(1-\lambda)\max_{1\le i\le n-2}\{\delta u_iu_{n-i}\}/(\delta 2^{n\alpha})\cr
&\le(u_{n-1}-\lambda u_{n-2})/(\delta 2^{n\alpha})\cr
&=\kappa(1/2^{n-1})^\beta/(\delta 2^\alpha).\end{array}$$
Then (iii) follows from $\kappa(1/2^n)\le L\kappa(1/2^{n-1})/2$.
\end{proof}

\begin{theorem}
For any $\beta>1$, there is a set of continuous function
$$\{f_\xi:[0,1]\to\Bbb R^+:\xi\in\{0,1\}^\omega\}$$
such that each {\bf E}$_{f_\xi}$ is equivalence relation with
$\Bbb R^\omega/\ell_1\le_B{\bf E}_{f_\xi}\le_B\Bbb R^\omega/\ell_\beta$, and
for and distinct $\xi,\zeta\in\{0,1\}^\omega$, we have {\bf E}$_{f_\xi}$ and {\bf E}$_{f_\zeta}$ are Borel incomparable.
\end{theorem}

\begin{proof}
Fix a $0<\delta<1$ and a $1\le\alpha<\beta$. Let $\lambda=2^{\alpha-\beta}$. For $s\in\{0,1\}^{<\omega}$, we denote by lh$(s)$ the length of $s$.
We are going to construct a finite decreasing sequence $w_s\in[0,1]^{<\omega}$, a natural number $n_s<\omega$
for every $s\in\{0,1\}^{<\omega}$, and a sequence of natural numbers $k_0<k_1<k_2<\cdots$, satisfying the following list of requirements.
\begin{enumerate}
\item[(a)] If lh$(s)=l$, then lh$(w_s)=k_l$.
\item[(b)] If $t|l=s$, then $w_t|k_l=w_s$.
\item[(c)] If lh$(s)=$lh$(t)=l,s\ne t$, then $k_{l-1}\le n_s<k_l$ and
$$w_s(n_s)/w_t(n_s)<1/2^l.$$
\end{enumerate}

Construct by induction on lh$(s)$. Firstly, let $k_0=2$, $w_\emptyset(0)=w_\emptyset(1)=1$ and $n_\emptyset=1$. Assume that $k_0<k_1<\cdots<k_{l-1}$
and for all lh$(s)<l$, $w_s,n_s$ have been defined. For lh$(s)=l$ and $n<k_{l-1}$, set $w_s(n)=w_{s|(l-1)}(n)$.

We enumerate $\{0,1\}^l$ by $s_1,s_2,\cdots,s_M\,(M=2^l)$.  Let $n_{s_1}$ be a sufficiently large number specified later,
for $s\in\{0,1\}^l,k_{l-1}\le n\le n_{s_1}$, we define
$$w_s(n)=\left\{\begin{array}{ll}\lambda w_s(n-1)+(1-\lambda)\max_{1\le i\le n-1}\{\delta w_s(i)w_s(n-i)\}, &s=s_1,\cr w_s(n-1), &s\ne s_1.\end{array}\right.$$
From Lemma \ref{un}, we see that $w_s$ is decreasing. Note that $w_{s_1}(i)w_{s_1}(2n-i)\le w_{s_1}(n)$ for $1\le i\le 2n-1$, we have
$$w_{s_1}(2n)\le\lambda w_{s_1}(n)+(1-\lambda)\delta w_{s_1}(n)=\delta'w_{s_1}(n),$$
in which $\delta'=\lambda+(1-\lambda)\delta<1$. Hence $w_{s_1}(2^mn)\le(\delta')^mw_{s_1}(n)\to 0\,(m\to\infty)$.
We can find a sufficient large $n_{s_1}$ such that, for $s_1\ne s\in\{0,1\}^l$,
$$w_{s_1}(n_{s_1})/w_s(n_{s_1})<1/2^l.$$
Follow the same method, we can find $n_{s_1}<n_{s_2}<\cdots<n_{s_M}$ such that, for $j=2,\cdots,M$ and $n_{s_{j-1}}<n\le n_{s_j}$,
$$w_s(n)=\left\{\begin{array}{ll}\lambda w_s(n-1)+(1-\lambda)\max_{1\le i\le n-1}\{\delta w_s(i)w_s(n-i)\}, &s=s_j,\cr w_s(n-1), &s\ne s_j.\end{array}\right.$$
Furthermore, for $s_j\ne s\in\{0,1\}^l$ we have
$$w_{s_j}(n_{s_j})/w_s(n_{s_j})<1/2^l.$$
Letting $k_l=n_{s_M}+1$, we finish the construction at level $l$.

For every $\xi\in\{0,1\}^\omega$, we fix a continuous increasing function $\varphi_\xi:[0,1]\to\Bbb R^+$ such that $\varphi_\xi(1/2^n)=w_{\xi|l}(n)$ for $l<\omega,n<k_l$.
Define $f_\xi(x)=x^\alpha\varphi_\xi(x)$ for $x\in[0,1]$. From Lemma \ref{reduce}, {\bf E}$_{f_\xi}$ is equivalence relation, and
$$\Bbb R^\omega/\ell_1\le_B{\bf E}_{f_\xi}\le_B\Bbb R^\omega/\ell_\beta.$$

By Lemma \ref{varphi}, condition (A$_1$) in Theorem \ref{inreduce} holds for every $\varphi_\xi$. If $\xi\ne\zeta$, then there exists $m$ such that $\xi(m)\ne\zeta(m)$.
Let $l>m,s=\zeta|l,t=\xi|l$, we have $s\ne t$. Then
$$\varphi_\zeta(1/2^{n_s})/\varphi_\xi(1/2^{n_s})=w_s(n_s)/w_t(n_s)<1/2^l.$$
We see that condition (A$_2$)' holds. By Remark \ref{incomparable}, we have {\bf E}$_{f_\zeta}\not\le_B${\bf E}$_{f_\xi}$.
\end{proof}

\begin{remark}
Let $1<\alpha<\beta$, we do not know whether there exist Borel functions $f,g:[0,1]\to\Bbb R^+$ such that {\bf E}$_f$,{\bf E}$_g$ are Borel incomparable
equivalence relations with $\Bbb R^\omega/\ell_\alpha\le_B{\bf E}_f,{\bf E}_g\le_B\Bbb R^\omega/\ell_\beta$.
\end{remark}


\begin{thebibliography}{99}

\bibitem{DH} R. Dougherty and G. Hjorth, \textit{Reducibility and
nonreducibility between $\ell^p$ equivalence relations}, Trans.
Amer. Math. Soc. 351 (1999) 1835--1844.

\bibitem{gao} S. Gao, \textit{Equivalence relations and classical Banach spaces}, Mathematical logic in Asia, World Sci. Publ., Hackensack, NJ, (2006), 70--89.

\bibitem{gaobook} S. Gao, Invariant Descriptive Set Theory,
Monographs and Textbooks in Pure and Applied Mathematics, vol. 293, CRC Press, 2008.

\bibitem{kechris} A. S. Kechiris, Classical Descriptive Set Theory, Graduate Texts in Mathematics, vol. 156, Springer-Verlag, 1995.

\bibitem{matrai} T. M\' atrai, \textit{On $\ell_p$-like equivalence relations}, Real Anal. Exchange 34 (2008/09), no. 2, 377--412.

\end{thebibliography}
\end{document}